\documentclass{amsart}
\usepackage{textcomp} 
\usepackage{indentfirst} 
\usepackage{amsmath}
\usepackage{amsthm}
\usepackage{amsfonts}
\usepackage{amssymb}
\usepackage{mathrsfs}
\usepackage{stmaryrd}
\usepackage{enumerate}
\usepackage{bbm}
\usepackage[T1]{fontenc}

\newtheorem{thm}{Theorem}[section]

\newtheorem{cor}[thm]{Corollary}

\newtheorem{prop}[thm]{Proposition}
\newtheorem*{thDWT}{Devinatz-Widom Theorem}
\theoremstyle{definition}

\newtheorem{rem}[thm]{Remark}
\newtheorem{qe}[thm]{Question}
\numberwithin{equation}{section}

\begin{document}
\title{ Examples of de Branges-Rovnyak spaces generated by nonextreme functions}
\author{Bartosz {\L}anucha, Maria T. Nowak}

\address{
Bartosz {\L}anucha,  \newline Institute of Mathematics,
\newline Maria Curie-Sk{\l}odowska University, \newline pl. M.
Curie-Sk{\l}odowskiej 1, \newline 20-031 Lublin, Poland}
\email{bartosz.lanucha@poczta.umcs.lublin.pl}

\address{
Maria T. Nowak,  \newline Institute of Mathematics,
\newline Maria Curie-Sk{\l}odowska University, \newline pl. M.
Curie-Sk{\l}odowskiej 1, \newline 20-031 Lublin, Poland}
\email{mt.nowak@poczta.umcs.lublin.pl}


\begin{abstract} We describe de Branges-Rovnyak spaces $\mathcal
H (b_{\alpha})$, $\alpha>0$, where the function $b_{\alpha}$ is not
extreme in the unit ball of $H^{\infty}$ on the unit disk $\mathbb D
$, defined by the equality $b_{\alpha}(z)/a_{\alpha}(z)=
(1-z)^{-\alpha}$, $z\in\mathbb D$, where  $a_{\alpha}$  is the
outer function such that $a_{\alpha}(0)>0$ and
$|a_{\alpha}|^2+|b_{\alpha}|^2= 1$ a.e. on $\partial \mathbb D$.
\end{abstract}

\maketitle
\section{Introduction}

Let $H^2$ denote the standard Hardy space in the open unit disk
$\mathbb D$ and let $\mathbb T=\partial \mathbb D$. For $\chi\in
L^{\infty}(\mathbb{T})$ let $T_{\chi}$ denote the bounded Toeplitz
operator on $H^2$, that is, $T_{\chi}f=P_{+}(\chi f)$, where $P_{+}$
is the orthogonal projection of $L^{2}(\mathbb{T})$ onto $H^2$. In
particular, $S=T_{e^{it}}$ is called the shift operator. We will
denote by $\mathcal{M}(\chi)$ the range of $T_{\chi}$ equipped with
the range norm, that is, the norm that makes the operator $T_{\chi}$
a coisometry of $H^2$ onto $\mathcal{M}(\chi)$.

Given a function $b$ in the unit ball of $H^{\infty}$, the
\textit{de Branges-Rovnyak space $\mathcal{H}(b)$} is the image of
$H^2$ under the operator $(I-T_bT_{\overline{b}})^{1/2}$ with the
corresponding range norm $\|\cdot\|_b$. 

 It is known that $\mathcal{H}(b)$ is a Hilbert space with
reproducing kernel
$$k_w^b(z)=\frac{1-\overline{b(w)}b(z)}{1-\overline{w}z}\quad(z,w\in\mathbb{D}).$$

Here we are interested in the case when the function $b$ is not an
extreme point of the unit ball of $H^{\infty}$. Then there exists an
outer function $a\in H^{\infty}$ for which $|a|^2+|b|^2=1$ a.e. on
$\mathbb{T}$. Moreover, if we suppose that $a(0)>0$, then $a$ is
uniquely determined, and, following Sarason,  we say that $(b,a)$ is
a \emph{pair}. The function $a$ is sometimes called the
\emph{Pythagorean mate} associated with $b$.

It is known that both $\mathcal{M}(a)$ and
$\mathcal{M}(\overline{a})$ are contained contractively in
$\mathcal{H}(b)$ (see \cite[p. 25]{sarason}). Moreover, if $(b,a)$
is a corona pair, that is, $|a|+|b|$ is bounded away from $0$ in
$\mathbb{D}$, then $\mathcal{H}(b)=\mathcal{M}(\overline{a})$ (see
e.g. \cite[p. 62]{sarason}).

Let us recall  that the Smirnov class $\mathcal{N}^{+}$ consists of
those holomorphic functions in $\mathbb{D}$ that are quotients of
functions in $H^{\infty}$ in which the denominators are outer
functions. If $(b,a)$ is a pair, then the quotient $\varphi=b/a$ is
in $\mathcal{N}^{+}$, and conversely, for every nonzero function
$\varphi\in\mathcal{N}^{+}$ there exists a unique pair $(b,a)$ such
that $\varphi=b/a$ (\cite{sarason2}).

 Many properties of $\mathcal{H}(b)$ can be expressed in
terms of the function $\varphi=b/a$ in the Smirnov class
$\mathcal{N}^{+}$. It is worth noting here that if $\varphi$ is
rational, then the functions $a$ and $b$ in the representation of
$\varphi$ are also rational (see \cite{sarason2}) and in such a case
$ (b,a)$ is called  a rational pair. Recently  spaces $\mathcal
H(b)$ for rational pairs have been studied in \cite{ransford2},
\cite{ross} and  \cite{LN}. In   \cite{ross} the authors described
also the spaces $\mathcal H (b^r)$, where $b$ is a rational outer
funtion in the closed unit ball of $H^{\infty}$ and $r$ is a
positive number.

 Here
we describe the Branges-Rovnyak spaces $\mathcal{H}(b_{\alpha})$,
$\alpha>0$, where $(b_{\alpha},a_{\alpha})$ is such a pair that
$$\varphi_{\alpha}(z)= \frac{b_{\alpha}(z)}{a_{\alpha}(z)}= \frac1{(1-z)^{\alpha}}$$
(principal branch).

For a function $\varphi$ that is holomorphic on $\mathbb{D}$ we
define $T_{\varphi}$ to be the operator of multiplication by
$\varphi$ on the domain $\mathcal{D}(T_{\varphi})=\{f\in H^2\colon\
\varphi f\in H^2\}$. It is well known that $T_{\varphi}$ is bounded
on $H^2$ if and only if $\varphi\in H^{\infty}$. Moreover, it was
proved in \cite{sarason2} that the domain $\mathcal{D}(T_{\varphi})$
is dense in $H^2$ if and only if $\varphi\in \mathcal{N}^{+}$. More
precisely, if $\varphi$ is a nonzero function in $\mathcal{N}^{+}$
with canonical representation $\varphi={b}/{a}$, then
$\mathcal{D}(T_{\varphi})=aH^2$. In this case $T_{\varphi}$ has a
unique, densely defined adjoint $T_{\varphi}^{*}$. In what follows
we denote $T_{\overline{\varphi}}=T_{\varphi}^{*}$ (see \cite[p.
286]{sarason2} for more details). The next theorem says that the
domain of $T_{\overline{\varphi}}$ coincides with the de
Branges-Rovnyak space $\mathcal{H}(b)$.

\begin{thm}[\cite{sarason2}]\label{sarsar}
Let $(b,a)$ be a pair and let $\varphi=b/a$. Then the domain of
$T_{\overline{\varphi}}$ is $\mathcal{H}(b)$ and for
$f\in\mathcal{H}(b)$,
$$\|f\|_{b}^2=\|f\|_{2}^2+\|T_{\overline{\varphi}}f\|_{2}^2.$$
\end{thm}

The next proposition was also proved in \cite{sarason2}.

\begin{prop}[\cite{sarason2}]\label{propek}
If $\varphi$ is in $\mathcal{N}^{+}$, $\psi$ is in $H^{\infty}$, and $f$ is in $\mathcal{D}(T_{\overline{\varphi}})$, then
$$T_{\overline{\varphi}}T_{\overline{\psi}}f=T_{\overline{\varphi}\overline{\psi}}f=T_{\overline{\psi}}T_{\overline{\varphi}}f.$$
\end{prop}

\begin{cor}\label{korek}
Let $\varphi_1,\varphi_2\in\mathcal{N}^{+}$ have canonical
representations $\varphi_{i}=b_{i}/a_{i}$, $i=1,2$. If
$\varphi_2/\varphi_1\in H^{\infty}$, then
$\mathcal{H}(b_1)\subset\mathcal{H}(b_2)$.
\end{cor}
\begin{proof}
Put $\psi = \varphi_2/\varphi_1$. It follows from Proposition
\ref{propek} that
$\mathcal{D}(T_{\overline{\varphi}_1})\subset\mathcal{D}(T_{\overline{\varphi}_1\overline{\psi}})$,
and so
$$\mathcal{H}(b_1)=\mathcal{D}(T_{\overline{\varphi}_1})\subset\mathcal{D}(T_{\overline{\varphi}_1\overline{\psi}})=\mathcal{D}(T_{\overline{\varphi}_2})=\mathcal{H}(b_2).$$
\end{proof}

In the proof of our main theorem we will use the following
description of invertible Toeplitz operators with unimodular
symbols.

\begin{thDWT}[\cite{nik}, p. 250]\label{animal}
Let $\psi\in L^{\infty}(\partial\mathbb{D})$ be such that $|\psi|=1$
a.e. on $\partial\mathbb{D}$. The following are equivalent.
\begin{itemize}
\item[(a)] $T_{\psi}$ is invertible.
\item[(b)] $\mathrm{dist}(\psi,H^{\infty})<1$ and $\mathrm{dist}(\overline{\psi},H^{\infty})<1$.
\item[(c)] There exists an outer function $h\in H^{\infty}$ such that $\|\psi-h\|_{\infty}<1$.
\item[(d)] There exist real valued bounded functions $u$, $v$ and a constant $c\in\mathbb{R}$ such that $\psi=e^{i(u+\tilde{v}+c)}$ and $\|u\|_{\infty}<\frac{\pi}2$, where $\tilde{v}$ denotes the conjugate function of $v$.
\end{itemize}
\end{thDWT}

We will need also the notion of a rigid function in $H^1$. A
function in $H^1$ is called rigid if no other functions in $H^1$,
except for positive scalar multiples of itself, have the same
argument as it almost everywhere on $\partial \mathbb{ D}$. As
observed in \cite{sarasonk}, every rigid function is outer. It is
known that the function $(1-z)^{\alpha}$ is rigid if $0<\alpha\leq
1$ and is not rigid if $\alpha>1$ (see e.g. \cite[Section 6.8]{fm}).

The next theorem shows a close connection between kernels of
Toeplitz operators and rigid functions in $H^1$ (\cite[p.
70]{sarason}).

\begin{thm}\label{jdr}
If $f$ is an outer function in $H^2$, then $f^2$ is rigid if and
only if the operator $T_{\overline{f}/f}$ has a trivial kernel.
\end{thm}

Moreover, for a pair $(b,a)$ the following  sufficient condition for
density of $\mathcal{M}(a)$ in $\mathcal{H}(b)$  is known (\cite[p.
72]{sarason}, \cite[vol. 2, p. 496]{fm}).
\begin{thm}\label{rigid} If the function $a^2$ is rigid,  then $\mathcal{M}(a)$ is dense in
$\mathcal{H}(b)$.
\end{thm}

\section{The spaces $\mathcal{H}(b_{\alpha})$, $\alpha>0$}
Recall that for  $\alpha>0$ we define the pair
$(b_{\alpha},a_{\alpha})$ by
$$\varphi_{\alpha}(z)=\frac{b_{\alpha}(z)}{a_{\alpha}(z)}=\frac1{(1-z)^{\alpha}}.$$

Consequently, the outer function $a_{\alpha}$ is given by
\begin{equation}\label{a}
a_{\alpha}(z)=\exp{\left\{\frac1{4\pi}\int_{0}^{2\pi}\frac{e^{it}+z}{e^{it}-z}\log{\frac{|1-e^{it}|^{2\alpha}}{1+|1-e^{it}|^{2\alpha}}}dt\right\}}.
\end{equation}
Since both $a_{\alpha}$ and $(1-z)^{\alpha}$ are outer functions,
the equality $(1-z)^{\alpha}b_{\alpha}(z)=a_{\alpha}(z)$ implies
that $b_{\alpha}$ is also outer. Hence
\begin{equation}\label{b}
b_{\alpha}(z)=a_{\alpha}(z)\varphi_{\alpha}(z)=\exp{\left\{\frac1{4\pi}\int_{0}^{2\pi}\frac{e^{it}+z}{e^{it}-z}\log{\frac{1}{1+|1-e^{it}|^{2\alpha}}}dt\right\}}.
\end{equation}

This formula shows that $\log{|b_{\alpha}(z)|}$ is a function
harmonic in ${\mathbb{D}}$ and continuous in
$\overline{\mathbb{D}}$. Moreover, $|b_{\alpha}(1)|=1$. We now prove
that actually $b_{\alpha}(1)=1$. To this end, it is enough to note
that $\mathrm{arg}b_{\alpha}(r)=0$ for all $0<r<1$. Indeed,
\begin{displaymath}
\begin{split}
\mathrm{arg}b_{\alpha}(r)&=\frac1{4\pi}\int_{0}^{2\pi}\mathrm{Im}\left(\frac{e^{it}+r}{e^{it}-r}\right)\log{\frac{1}{1+|1-e^{it}|^{2\alpha}}}dt\\
&=-\frac1{4\pi}\int_{-\pi}^{\pi}\frac{2r\sin t
}{|e^{it}-r|^2}\log{\frac{1}{1+|1-e^{it}|^{2\alpha}}}dt=0,
\end{split}
\end{displaymath}
because the integrand is an odd function.

The following proposition says for which $\alpha$ a nontangential
limit at 1 of each function (and its derivatives up to a given
order) from $\mathcal{H}(b_{\alpha})$ exists.

\begin{prop}\label{deri}
Let $n\in\mathbb{N}$. Every $f\in\mathcal{H}(b_{\alpha})$ along with
its derivatives up to order $n-1$ has a nontangential limit at the
point $1$ if and only if $\alpha>n-1/2$.
\end{prop}

This is a consequence of Theorem 3.2 from \cite{fric} (see also
\cite{sarason} and \cite{ross}), which states that the following two
conditions are equivalent:
\begin{itemize}
\item[(i)] for every $f\in\mathcal{H}(b_{\alpha})$ the functions $f(z), f'(z),\ldots , f^{(n-1)}(z)$ have finite limits as $z$
tends nontangentially to $1$;
\item[(ii)]
$$\int_{0}^{2\pi}\frac{|\log|b_{\alpha}(e^{it})||}{|1-e^{it}|^{2n}}dt<+\infty.$$
\end{itemize}

Since

\begin{displaymath}
\log{|b_{\alpha}(e^{it})|^2}=\log{\frac{1}{1+|1-e^{it}|^{2\alpha}}}=\log{\left(1-\frac{|1-e^{it}|^{2\alpha}}{1+|1-e^{it}|^{2\alpha}}\right)}
\end{displaymath}
and $|\log{(1-x)}|\approx |x|$ for $x$ sufficiently close to zero,
we have
\begin{displaymath}
\log{|b_{\alpha}(e^{it})|}\approx\frac{|1-e^{it}|^{2\alpha}}{1+|1-e^{it}|^{2\alpha}}\approx
|1-e^{it}|^{2\alpha}
\end{displaymath}
whenever $t$ is sufficiently close to $0$ or $2\pi$. This implies
that
\begin{displaymath}
\int_0^{2\pi}\frac{|\log{|b_{\alpha}(e^{it})|}|}{|1-e^{it}|^{2n}}dt<\infty
\end{displaymath}
if and only if
\begin{displaymath}
\int_0^{2\pi}\frac{1}{|1-e^{it}|^{2n-2\alpha}}dt<\infty,
\end{displaymath}
which holds only when $\alpha>n-1/2$.

In particular, we see that every $f\in\mathcal{H}(b_{\alpha})$ has a
nontangential limit at $1$ if and only if $\alpha>1/2$.

The next proposition is an immediate consequence of Corollary
\ref{korek}.
\begin{prop}
For every $0<\alpha\leq\beta<\infty$,
$$\mathcal{H}(b_{\beta})\subset \mathcal{H}(b_{\alpha}).$$
\end{prop}

Finally, we observe that
$$|b_{\alpha}(z)|\geq \sqrt{\frac1{1+4^{\alpha}}},$$
which implies that $(b_{\alpha},a_{\alpha})$ is a corona pair for
$\alpha>0$.

\begin{cor}\label{opiec}
For $\alpha>0$,
$$\mathcal{M}({a}_{\alpha})=\mathcal{M}((1-z)^{\alpha})\quad\mathrm{and}\quad\mathcal{H}(b_{\alpha})=\mathcal{M}(\overline{a}_{\alpha})=\mathcal{M}(\overline{(1-z)^{\alpha}})$$
with equivalence of norms.
\end{cor}
\begin{proof}
The equality of $\mathcal{H}(b_{\alpha})$ and
$\mathcal{M}(\overline{a}_{\alpha})$ follows from the fact that
$(b_{\alpha},a_{\alpha})$ is a corona pair, which in turn is a
consequence of the fact that $b_{\alpha}$ is bounded below. The
latter implies that $1/b_{\alpha}\in H^{\infty}$ and so
$T_{b_{\alpha}}$ and $T_{\overline{b}_{\alpha}}$ are invertible.
Hence
$$\mathcal{M}((1-z)^{\alpha})=T_{\frac{{a}_{\alpha}}{{b}_{\alpha}}}H^2=T_{{a}_{\alpha}}H^2$$
and
$$\mathcal{M}(\overline{(1-z)^{\alpha}})=T_{\frac{\overline{a}_{\alpha}}{\overline{b}_{\alpha}}}H^2=T_{\overline{a}_{\alpha}}H^2.$$
Both $\mathcal{M}({a}_{\alpha})$ and $\mathcal{M}((1-z)^{\alpha})$
are boundedly contained in $H^2$. Hence, the Closed Graph Theorem
implies equivalence of their norms. Similarly, one obtains the
equivalence of norms in $\mathcal{M}(\overline{a}_{\alpha})$ and
$\mathcal{M}(\overline{(1-z)^{\alpha}})$.
\end{proof}

\section{Main results}
We start with the following.

\begin{thm}\label{lemBB}
For any $n\in\mathbb{N}$ and $n-1/2<\alpha<n+1/2$ we have
\begin{equation*}
\mathcal{M}(\overline{(1-z)^{\alpha}})=\mathcal{M}((1-z)^{\alpha})+\mathrm{span}\{S^{*}(1-z)^{\alpha},\ldots,
S^{*n}(1-z)^{\alpha}\}.
\end{equation*}
\end{thm}
\begin{proof}
Let $$Q(z)=\frac{1-z}{\overline{1-z}},\quad z\in\mathbb{D}.$$ Then
$Q$ has a continuous extension to
$\overline{\mathbb{D}}\setminus\{1\}$ and
$$Q(e^{it})=e^{(t-\pi)i},\quad t\in (0,2\pi),$$ which implies that $$T_{Q^n}=(-1)^nS^n\quad \mathrm{for}\ n\geq1.$$

Moreover, we observe that for $n-1/2<\alpha<n+1/2$, $n\geq 1$, we
have
$$T_{Q^{\alpha}}=T_{Q^{\alpha-n}Q^n}=(-1)^nT_{Q^{\alpha-n}}S^n.$$
Consequently,
\begin{equation}\label{krop}
T_{(1-z)^{\alpha}}=T_{\overline{(1-z)^{\alpha}}Q^{\alpha}}=(-1)^nT_{\overline{(1-z)^{\alpha}}}T_{Q^{\alpha-n}}S^n.
\end{equation}

Observe now that the operator $T_{Q^{\alpha-n}}$ is invertible. This
is an immediate consequence of the Devinatz-Widom Theorem.

Let $f\in \mathcal{M}(\overline{(1-z)^{\alpha}})$ and
$f=T_{\overline{(1-z)^{\alpha}}}g$ for a function $g\in H^2$. Since
$T_{Q^{\alpha-n}}$ is invertible, there exists $g_0\in H^2$ such
that $(-1)^ng= T_{Q^{\alpha-n}}g_0$. Hence, using \eqref{krop}, we obtain
\begin{displaymath}
\begin{split}
f=T_{\overline{(1-z)^{\alpha}}}g&=(-1)^nT_{\overline{(1-z)^{\alpha}}}T_{Q^{\alpha-n}}g_0\\
&= (-1)^nT_{\overline{(1-z)^{\alpha}}}T_{Q^{\alpha-n}}\left(S^nS^{*n}g_0+\sum_{k=0}^{n-1}\langle g_0,z^k\rangle z^k\right)\\
&= T_{(1-z)^{\alpha}}S^{*n}g_0+(-1)^n\sum_{k=0}^{n-1}\langle
g_0,z^k\rangle T_{\overline{(1-z)^{\alpha}}}T_{Q^{\alpha-n}}z^k.
\end{split}
\end{displaymath}
Since for $0\leq k\leq n-1$,
\begin{equation}\label{krop2}
\begin{split}
(-1)^nT_{\overline{(1-z)^{\alpha}}}T_{Q^{\alpha-n}}z^k&=(-1)^nT_{\overline{Q}^n(1-z)^{\alpha}}S^k1\\&=S^{*(n-k)}T_{(1-z)^{\alpha}}1=S^{*(n-k)}(1-z)^{\alpha},
\end{split}
\end{equation}
we get
\begin{displaymath}
\begin{split}
f=& (1-z)^{\alpha}S^{*n}g_0+\sum_{k=0}^{n-1}\langle g_0,z^k\rangle S^{*(n-k)}(1-z)^{\alpha}\\
&\in
\mathcal{M}((1-z)^{\alpha})+\mathrm{span}\{S^{*}(1-z)^{\alpha},\ldots,
S^{*n}(1-z)^{\alpha}\}.
\end{split}
\end{displaymath}

On the other hand, if
$$f= (1-z)^{\alpha}h+\sum_{k=1}^{n}c_k S^{*k}(1-z)^{\alpha},\quad h\in H^2,$$
then, by \eqref{krop} and \eqref{krop2},
\begin{displaymath}
\begin{split}
f=& T_{(1-z)^{\alpha}}h+\sum_{k=0}^{n-1}c_{n-k} S^{*(n-k)}(1-z)^{\alpha} \\
=&(-1)^nT_{\overline{(1-z)^{\alpha}}}T_{Q^{\alpha-n}}S^nh+(-1)^n\sum_{k=0}^{n-1}c_{n-k} T_{\overline{(1-z)^{\alpha}}}T_{Q^{\alpha-n}}z^k\\
=&T_{\overline{(1-z)^{\alpha}}}\left((-1)^nT_{Q^{\alpha-n}}S^nh+(-1)^n\sum_{k=0}^{n-1}c_{n-k}
T_{Q^{\alpha-n}}z^k\right)\in
\mathcal{M}(\overline{(1-z)^{\alpha}}).
\end{split}
\end{displaymath}

\end{proof}

Now we prove our main result.
\begin{thm}\label{mejn}
Let $0<\alpha<\infty$ and let $(b_{\alpha}, a_{\alpha})$ be a pair,
with the functions $b_{\alpha}$ and $a_{\alpha}$ given by \eqref{b}
and \eqref{a}, respectively. Then
\begin{enumerate}[(i)]
\item for $0<\alpha<1/2$,
$$\mathcal{H}(b_{\alpha})=\mathcal{M}(a_{\alpha})=(1-z)^{\alpha}H^2,$$
\item for $n-1/2<\alpha<n+1/2$, $n=1,2,\ldots$,
$$\mathcal{H}(b_{\alpha})=\mathcal{M}(a_{\alpha})+\mathcal{P}_n=(1-z)^{\alpha}H^2+\mathcal{P}_n,$$
where $\mathcal{P}_n$ is the set of all polynomials of degree at most
$n-1$,
\item
$$\mathcal{H}(b_{1/2})=\overline{\mathcal{M}(a_{1/2})}=\overline{(1-z)^{1/2}H^2},$$
where the closure is taken with respect to the
$\mathcal{H}(b_{1/2})$-norm,
\item for $\alpha=n+1/2$, $n=1,2,\ldots$,
$$\mathcal{H}(b_{\alpha})=\overline{\mathcal{M}(a_{\alpha})}+\mathcal{A}_n,$$
where the closure is taken with respect to the
$\mathcal{H}(b_{\alpha})$-norm and $\mathcal{A}_n$ is the
$n$-dimensional subspace of $\mathcal{H}(b_{\alpha})$ defined by
$$\mathcal{A}_n=\left\{ p_n\cdot
P_+\left(\overline{(1-z)^{\alpha}}{(1-z)}^{1/2}\right)+P_+\left(p_nP_-\left(\overline{(1-z)^{\alpha}}{(1-z)}^{1/2}\right)\right)\
\colon \ p_n\in \mathcal{P}_n \right\},$$ where $P_-=I-P_+$.
\end{enumerate}

\end{thm}


\begin{proof}(i) We know from Corollary \ref{opiec} that for $\alpha>0$,
$$\mathcal{H}(b_{\alpha})=\mathcal{M}(\overline{a}_{\alpha})=\mathcal{M}(\overline{(1-z)^{\alpha}}).$$
We first observe that for $0<\alpha<1/2$ the operator
$T_{(1-z)^{\alpha}/\overline{(1-z)^{\alpha}}}$ is invertible. This
follows from
$$\frac{(1-e^{it})^{\alpha}}{\overline{(1-e^{it})}^{\alpha}}=e^{i\alpha(t-\pi)},\quad t\in(0,2\pi),$$
and the Devinatz-Widom Theorem.

Consequently,
$$\mathcal{M}(\overline{(1-z)^{\alpha}})=T_{\overline{(1-z)^{\alpha}}}H^2=T_{\overline{(1-z)^{\alpha}}}T_{\frac{(1-z)^{\alpha}}{\overline{(1-z)^{\alpha}}}}H^2=(1-z)^{\alpha}H^2.$$

(ii) Since $\mathcal{H}(b_{\alpha})$ contains
$\mathcal{M}(a_{\alpha})=\mathcal{M}((1-z)^{\alpha})$ and all
polynomials (see e.g. \cite[p. 25]{sarason}), to prove (ii) it is
enough to show that

$$\mathcal{H}(b_{\alpha})\subset\mathcal{P}_n+\mathcal{M}((1-z)^{\alpha}).$$

By Theorem \ref{lemBB} we have
$$\mathcal{H}(b_{\alpha})=\mathcal{M}(\overline{(1-z)^{\alpha}})=\mathcal{M}((1-z)^{\alpha})+\mathrm{span}\{S^{*}(1-z)^{\alpha},\ldots, S^{*n}(1-z)^{\alpha}\}.$$

Therefore, we only need to show that

\begin{displaymath}
\mathrm{span}\{S^{*}(1-z)^{\alpha},\ldots,
S^{*n}(1-z)^{\alpha}\}\subset\mathcal{P}_n+\mathcal{M}((1-z)^{\alpha}).
\end{displaymath}

Clearly,
\begin{displaymath}
\begin{split}
S^{*}(1-z)^{\alpha}&=\frac{(1-z)^{\alpha}-1}{z}=\frac{(1-z)^{\alpha}-(1-z)^n+(1-z)^n-1}{z}\\
&=S^{*}(1-z)^{n}-(1-z)^{\alpha}S^{*}(1-z)^{n-\alpha}\in
\mathcal{P}_n+\mathcal{M}((1-z)^{\alpha})
\end{split}
\end{displaymath}
($(1-z)^{n-\alpha}\in H^2$ since $n-\alpha>-1/2$). Now assume that
for any $1\leq k<n$,
$$S^{*k}(1-z)^{\alpha}\in \mathcal{P}_n+\mathcal{M}((1-z)^{\alpha}),$$
or, in other words,
$$S^{*k}(1-z)^{\alpha}=p_n+(1-z)^{\alpha}h_k\text{ for some }p_n\in \mathcal{P}_n\ \mathrm{and}\ h_k\in H^2.$$
Then
\begin{displaymath}
\begin{split}
S^{*(k+1)}(1-z)^{\alpha}&=S^{*}(S^{*k}(1-z)^{\alpha})=\frac{ p_n+(1-z)^{\alpha}h_k-p_n(0)-h_k(0)}{z}\\
&=\frac{ p_n+(1-z)^{\alpha}h_k-(1-z)^{\alpha}h_k(0)+(1-z)^{\alpha}h_k(0)-p_n(0)-h_k(0)}{z}\\
&= S^{*}p_n+h_k(0)S^{*}(1-z)^{\alpha} +(1-z)^{\alpha}S^{*}h_k\in
\mathcal{P}_n+\mathcal{M}((1-z)^{\alpha}).
\end{split}
\end{displaymath}
This completes the proof of (ii).

(iii) In view of Theorem \ref{rigid}, to prove (iii) it is enough to
show that $a_{1/2}^2$ is a rigid function. We actually prove that
$a_{\alpha}^2$ is rigid for every $0<\alpha\leq1/2$.

To this end, we  observe  that  for $\alpha>0 $,
\begin{equation}\label{inequal} \frac{1}{\sqrt{1+4^{\alpha}}}|1-z|^{\alpha} \leq
|a_{\alpha}(z)|\leq |1-z|^{\alpha},\quad  z\in\mathbb
D.\end{equation}
 This follows from \eqref{a} and the representation
of the outer function $$(1-z)^{\alpha}= \exp\left\{
\frac{\alpha}{2\pi }\int_0^{2\pi}
\frac{e^{it}+z}{e^{it}-z}\log|1-e^{it}|dt\right\}.$$ Thus we have
$$\frac{|a_{\alpha}(z)|}{|1-z|^{\alpha}}= \exp\left\{
\frac{1}{2\pi }\int_0^{2\pi} \frac{1-|z|^2}{|1-ze^{-it}|^2}\log\frac
1{\sqrt{1+|1-e^{it}|^{2\alpha}}}dt\right\}$$ which implies
inequalities \eqref {inequal}.

Now we use a reasoning analogous to that in \cite[(X--5)]{sarason}.
If $a_{\alpha}^2$ is not rigid for some $0<\alpha\leq1/2$, then by
Theorem \ref{jdr} there is a nonzero function $g$ in the kernel of
$T_{\overline{a}_{\alpha}/a_{\alpha}}$. Then
\begin{displaymath}
T_{\frac{\overline{(1-z)^{\alpha}}}{(1-z)^{\alpha}}}\left(\tfrac{(1-z)^{\alpha}g}{a_{\alpha}}\right)=P_+\left(\tfrac{\overline{(1-z)^{\alpha}}g}{a_{\alpha}}\right)=P_+\left(\tfrac{\overline{(1-z)^{\alpha}}g}{a_{\alpha}}\cdot
\tfrac{\overline{a}_{\alpha}}{\overline{a}_{\alpha}}\right)=T_{\frac{\overline{(1-z)^{\alpha}}}{\overline{a}_{\alpha}}}T_{\frac{\overline{a}_{\alpha}}{{a}_{\alpha}}}g=0,
\end{displaymath}
which means that $(1-z)^{\alpha}g/a_{\alpha}$ is a nonzero function
in the kernel of $T_{\overline{(1-z)^{\alpha}}/(1-z)^{\alpha}}$,
contrary to the fact that $(1-z)^{2\alpha}$ is rigid for
$0<\alpha\leq1/2$ (see, e.g., \cite[Section 6.8]{fm}).

(iv) We know that for every $\alpha>0$,
$$\mathcal{H}(b_{\alpha})=\mathcal{M}(\overline{a}_{\alpha})=\mathcal{M}(\overline{(1-z)^{\alpha}})=T_{\overline{(1-z)^{\alpha}}}H^2$$
and $\mathcal{M}(a_{\alpha})=\mathcal{M}((1-z)^{\alpha})$ is the
image under $T_{\overline{(1-z)^{\alpha}}}$ of the range of
$T_{(1-z)^{\alpha}/\overline{(1-z)^{\alpha}}}$, that is,
$$\mathcal{M}((1-z)^{\alpha})=T_{\overline{(1-z)^{\alpha}}}T_{\frac
{(1-z)^{\alpha}}{\overline{(1-z)^{\alpha}}}}H^2.
$$

It follows that the orthogonal complement of
$\mathcal{M}((1-z)^{\alpha})$ in the space
$\mathcal{M}(\overline{(1-z)^{\alpha}})$ is the image under
$T_{\overline{(1-z)^{\alpha}}}$ of $\ker T_{\overline{(1-z)^{\alpha}}/(1-z)^{\alpha}}$.

We now observe that for $\alpha=n+1/2$, $$\ker T_{\frac
{\overline{(1-z)^{\alpha}}}{(1-z)^{\alpha}}}=\ker
T_{\overline{z^n}}T_{\frac
{\overline{(1-z)^{1/2}}}{(1-z)^{1/2}}}=(1-z)^{1/2}\mathcal{P}_n,$$
where $\mathcal{P}_n$ is the set
 of all polynomials of degree at most $n-1$. Finally, note that if $p_n$ is in
 $\mathcal{P}_n$, then

 \begin{displaymath}
 \begin{split}
T_{\overline{(1-z)^{\alpha}}}\left((1-z)^{1/2}p_n
\right)&=P_+\left(\overline{(1-z)^{\alpha}}(1-z)^{1/2}p_n \right)=\\
&=P_+\left(\overline{(1-z)^{\alpha}}(1-z)^{1/2}
\right)p_n+P_+\left(P_{-}\left(\overline{(1-z)^{\alpha}}(1-z)^{1/2}\right)p_n
\right).
\end{split}
 \end{displaymath}
Our claim follows.
\end{proof}

The following corollary is just another statement of (ii) in Theorem
\ref{mejn}.

\begin{cor}
For any $n\in\mathbb{N}$ and $n-1/2<\alpha<n+1/2$ we have
\begin{equation*}
\mathcal{H}(b_{\alpha})=\mathcal{M}(a_{\alpha})+\mathcal{P}_n=\mathcal{M}(a_{\alpha})+\mathrm{span}\{T_{\overline{a}_{\alpha}}1,\ldots,
T_{\overline{a}_{\alpha}}z^{n-1}\}.
\end{equation*}
\end{cor}

\begin{rem}
We observe that since $a_{\alpha}^2$ is rigid for all $0<\alpha\leq
1/2$, Theorem \ref{rigid} implies that the space
$\mathcal{M}(a_{\alpha})$ is dense in $\mathcal{H}(b_{\alpha})$ for
all such $\alpha$. However, for $0<\alpha< 1/2$ we have
$\mathcal{M}(a_{\alpha})=\mathcal{H}(b_{\alpha})$, while
$\mathcal{M}(a_{1/2})\subsetneq\mathcal{H}(b_{1/2})$. The latter
follows from the fact that every $h\in H^2$ satisfies
$|h(z)|=o((1-|z|)^{1/2})$ as $|z|\rightarrow 1^{-}$. Thus if
$f\in\mathcal{M}(a_{1/2})$, then $f(z)=(1-z)^{1/2}h(z)$, $h\in H^2$,
and
$$|f(z)|=|1-z|^{\frac12}|h(z)|=\left(\frac{|1-z|}{1-|z|}\right)^{\frac12}|h(z)|(1-|z|)^{\frac12}.$$
This shows that the nontangential limit of $f$ at $1$ is $0$. On the
other hand, $\mathcal{H}(b_{1/2})$ contains nonzero constant
functions, so $\mathcal{M}(a_{1/2})$ cannot be equal to
$\mathcal{H}(b_{1/2})$.
\end{rem}

\begin{cor}
If $n-1/2<\alpha <n+1/2$, $n\in\mathbb{N}$, and
$f\in\mathcal{H}(b_{\alpha})$, then there is a function $h$ in $H^2$
such that
$$f(z)=f(1)+f'(1)(z-1)+\ldots+\frac{f^{(n-1)}(1)}{(n-1)!}(z-1)^{n-1}+(1-z)^{\alpha}h(z).$$
\end{cor}
\begin{proof}
It follows from Proposition \ref{deri} that $f$ and its derivatives
of order up to $n-1$ have nontangential limits at $1$, say
$f(1),f'(1),\ldots,f^{(n-1)}(1)$. By Theorem \ref{mejn}(ii), $f$ can
be written as
$$f(z)=p_n(z)+(1-z)^{\alpha}h(z)=\sum_{k=0}^{n-1}a_k(z-1)^k+(1-z)^{\alpha}h(z),\quad h\in
H^2.$$ Since every $h$ in $H^2$ satisfies $$|h^{(k)}(z)|\leq
\frac{c_k}{(1-|z|)^{k+\frac12}},$$ we find that
$$a_k=\frac{p_n^{(k)}(1)}{k!}=\frac{f^{(k)}(1)}{k!}\quad\text{for
}k=0,1,\ldots,n-1.$$
\end{proof}

The next theorem describes the space $\mathcal{H}(\tilde
{b}_{\alpha})$ where $\tilde {b}_{\alpha}$ is an outer function from
the unit ball of $H^{\infty}$ whose Pythagorean mate is $\left(\frac
{1-z}2\right)^{\alpha}$, $\alpha>0$.

\begin{thm}
For $\alpha>0$ let  $\tilde {a}_{\alpha}(z)=\left(\frac
{1-z}2\right)^{\alpha}$ and let $\tilde {b}_{\alpha}$ be  the outer
function such that $(\tilde {b}_{\alpha},\tilde {a}_{\alpha})$ is a
pair. Then
$$ \mathcal{H}(\tilde {b}_{\alpha})=\mathcal{H}(b_{\alpha})$$
\end{thm}
\begin{proof} It is enough to show that $(\tilde {b}_{\alpha}, \tilde
{a}_{\alpha})$ is  a corona pair. The function  $\tilde
{a}_{\alpha}$ is continuous on $\overline{\mathbb D}$  and vanishes
only at $1$. 
Since $|\tilde {b}_{\alpha}(1)|=\tilde {a}_{\alpha}(-1)=1,$ there
exist $\delta>0$ such that $|\tilde {b}_{\alpha}(z)|>1/2\ $  on
$D_1=\overline{\mathbb D}\cap\{z: |z-1|<\delta\}$ and $|\tilde
{a}_{\alpha}(z)|>1/2$  on $ D_2=\overline{\mathbb D}\cap\{z:
|z+1|<\delta\}$. Then the continuous function $|\tilde
{b}_{\alpha}|^2+|\tilde {a}_{\alpha}|^2$ is positive on the compact
set $\overline{\mathbb D}\setminus (D_1\cup D_2)$, so it is bounded
from below by a strictly positive number $\varepsilon>0$.
\end{proof}
\begin{rem}
Since $\tfrac{1-z}2$ is the Pythagorean mate for
$\tfrac{1+z}2$, we remark that it follows from \cite{ross} that for $\alpha>0$,
$$\mathcal{H}\left(\left(\tfrac {1+z}2\right)^{\alpha}\right)=\mathcal{H}\left(\tfrac {1+z}2\right)=c+(1-z)H^2$$
as sets.
\end{rem}

Finally, we remark that if $u$ is a finite Blaschke product and
$b_{\alpha}$ is given by \eqref{b}, then
\begin{equation}\label{bal}
\mathcal{H}(ub_{\alpha})=\mathcal{H}(b_{\alpha}).
\end{equation}
Since every function in $\mathcal{H}(u)$  is holomorphic  in
$\overline{\mathbb D}$ (see, e.g. \cite[Sec. 14.2]{fm})  and
$\mathcal{H}(b_{\alpha})$ is invariant under multiplication by
functions holomorphic in $\overline{\mathbb D}$ (see, e.g. \cite[
(IV-6)]{sarason}),  \eqref{bal} follows from the equality
$$
\mathcal{H}(ub_{\alpha})=\mathcal{H}(u)+u\mathcal{H}(b_{\alpha}).$$

\begin{qe} Can one
characterize all inner functions $u$ for which equality \eqref{bal}
holds?
\end{qe}

\end{document}